\documentclass[12pt,a4paper]{amsart}

\usepackage[utf8]{inputenc}
\usepackage[english]{babel}
\usepackage[T1]{fontenc}
\usepackage{bbm}
\usepackage{hyperref}
\usepackage{enumitem}

\usepackage[mathlines,pagewise]{lineno}
\newcommand*\patchAmsMathEnvironmentForLineno[1]{%
  \expandafter\let\csname old#1\expandafter\endcsname\csname #1\endcsname
  \expandafter\let\csname oldend#1\expandafter\endcsname\csname end#1\endcsname
  \renewenvironment{#1}%
     {\linenomath\csname old#1\endcsname}%
     {\csname oldend#1\endcsname\endlinenomath}}%
\newcommand*\patchBothAmsMathEnvironmentsForLineno[1]{%
  \patchAmsMathEnvironmentForLineno{#1}%
  \patchAmsMathEnvironmentForLineno{#1*}}%
\AtBeginDocument{%
\patchBothAmsMathEnvironmentsForLineno{equation}%
\patchBothAmsMathEnvironmentsForLineno{align}%
\patchBothAmsMathEnvironmentsForLineno{flalign}%
\patchBothAmsMathEnvironmentsForLineno{alignat}%
\patchBothAmsMathEnvironmentsForLineno{gather}%
\patchBothAmsMathEnvironmentsForLineno{multline}%
}

\setenumerate[1]{label=(\roman{*}), ref=\roman{*}}
\setenumerate[2]{label=(\alph{*}), ref=\alph{*}}

\theoremstyle{plain}
\newtheorem*{theorem*}{\indent Theorem}
\newtheorem{theorem}{\indent Theorem}[section]
\newtheorem{corollary}[theorem]{\indent Corollary}
\newtheorem{lemma}[theorem]{\indent Lemma}
\newtheorem{proposition}[theorem]{\indent Proposition}

\theoremstyle{definition}

\newtheorem*{example}{\indent Example}

\theoremstyle{remark}
\newtheorem{rem}[theorem]{\indent Remark}

\DeclareMathOperator{\sgn}{sign}
\DeclareMathOperator{\supp}{supp}
\DeclareMathOperator{\linspan}{span}

\newcommand{\iten}{\ensuremath{\widehat{\otimes}_\varepsilon}}
\newcommand{\pten}{\ensuremath{\widehat{\otimes}_\pi}}
\newcommand{\aten}{\ensuremath{\widehat{\otimes}_\alpha}}

\title[Octahedral norms in tensor products of Banach spaces]{Octahedral norms in tensor products of Banach spaces}

\author[J. Langemets]{Johann Langemets}
\address[J. Langemets]{Institute of Mathematics and Statistics, University of Tartu, J. Liivi 2, 50409 Tartu, Estonia} \email{johann.langemets@ut.ee}
\thanks{The research of J.~Langemets was supported by institutional research funding IUT20-57 of the Estonian Ministry of Education and Research.}

\author[V.~Lima]{Vegard Lima}
\address[V.~Lima] {NTNU, Norwegian University of Science and
  Technology, Aalesund, Postboks 1517, N-6025 {\AA}lesund Norway.}
\email{Vegard.Lima@ntnu.no}

\author[A. Rueda Zoca]{Abraham Rueda Zoca}
\address[A. Rueda Zoca]{Universidad de Granada, Facultad de Ciencias.
Departamento de An\'{a}lisis Matem\'{a}tico, 18071-Granada
(Spain)} \email{arz0001@correo.ugr.es}\thanks{The research of A.~Rueda Zoca was supported by a research grant Contratos predoctorales FPU del Plan
Propio del Vicerrectorado de Investigaci{\'o}n y Transferencia de la Universidad de Granada and by Junta de Andaluc\'ia Grants FQM-0199.}\urladdr{\url{https://arzenglish.wordpress.com}}


\begin{document}

\begin{abstract}
  We continue the investigation of the behaviour of octahedral norms in tensor
  products of Banach spaces. Firstly, we will prove the existence of a
  Banach space $Y$ such that the injective tensor products $l_1\iten Y$
  and $L_1\iten Y$ both fail to have an octahedral norm, which solves two open problems from the literature. Secondly, we will show that in the presence of the metric approximation property octahedrality is preserved from a non-reflexive $L$-embedded Banach space taking projective tensor products with an arbitrary Banach
  space. 
\end{abstract}

\maketitle

\section{Introduction}\label{sectionintro}
\bigskip

According to \cite[Remark~II.5.2]{G-octa}, the norm of a Banach space $X$ is
\emph{octahedral} if, for every finite-dimensional subspace $E$ of $X$
and every $\varepsilon > 0$, there exists $y \in S_X$ such that
\begin{equation*}
  \| x + \lambda y \| \ge (1-\varepsilon)(\|x\|+|\lambda|)
  \ \mbox{for every } x \in E\mbox{ and every }\lambda \in \mathbb{R}.
\end{equation*}

The starting point of dual characterisations of octahedral norms was in \cite{dev},
where the author proved that if a Banach space $X$ has an octahedral norm then the
dual $X^*$ enjoys the \textit{weak$^*$ strong diameter two property} ($w^*$-SD2P),
i.e. every convex combination of weak-star slices of the dual unit ball has diameter two.
The converse of this result was proved in \cite[Theorem~2.1]{blr0}
(see also \cite{hlp0, lan}). It follows that a Banach space has the
\textit{strong diameter two property} (SD2P)
(i.e. every convex combination of slices of the unit ball has diameter two)
if, and only if, the dual norm is octahedral.
This characterisation motivated a lot of research on octahedral norms
in connection with the so called ``big slice phenomenon'' and it
will be used repeatedly without reference throughout this text.

The connection between the SD2P and octahedrality was the basis for
new results related to the big slice phenomenon in tensor product
spaces and, by duality, in spaces of operators.  Indeed, in
\cite[Theorem~2.5]{blr} it was proved that given two Banach spaces $X$
and $Y$ such that the norms of $X^*$ and $Y$ are octahedral then the
norm of every closed subspace $H$ of $L(X,Y)$ which contains
finite-rank operators is octahedral. As a corollary, the projective
tensor product of two spaces having the SD2P enjoys the SD2P, a result
which improved the main results of \cite{abr} and gave a partial
answer to \cite[Question~(b)]{aln}, where it was asked how diameter
two properties are preserved by tensor product spaces. However, it
remained an open problem whether the assumption of the SD2P on one
of the factor can be eliminated \cite[p.~177]{blr}.  In
\cite[Theorem~2.2]{llr} a result similar to \cite[Theorem~2.5]{blr}
appeared, proving that octahedrality is preserved by taking injective
tensor products from both factors. But the question whether the
assumption on one of the factor can be removed remained open
\cite[Question~4.1]{llr} (see also \cite[p.~5]{hlp}).

Dually, it is a natural question how octahedrality is preserved by
projective tensor products.  There are several examples
\cite[Examples]{llr} which suggest that it should be sufficient to
assume octahedrality on one of the factors for the projective tensor
product to have an octahedral norm, and this was posed as an open
problem \cite[Question~4.4]{llr}. Even the particular case of
Lipschitz-free spaces have been considered \cite[Question~2]{blr2}.

The aim of this note is to continue studying octahedrality in tensor product spaces and to give some complete and some partial
answers to the above questions. We
start by giving definitions and preliminary results in Section
\ref{sectionpreli}.  In Section \ref{sectioninje} we will prove that
there are Banach spaces $Y$ such that the injective tensor products
$\ell_1\iten Y$ and $L_1\iten Y$ fail to have an octahedral norm.
Indeed, we will characterise in Theorem \ref{summary} when the spaces
$X\iten Y$ have an octahedral norm whenever $X$ is either $\ell_1$ or
$L_1$ and $Y$ is either $\ell_p$ or $\ell_p^n$ for $1\leq p\leq
\infty$ and $n\geq 2$.  This will give a negative answer to
\cite[Question~4.1]{llr} and to a question from \cite[p.~177]{blr}. Moreover, Theorem \ref{summary} also gives a complete answer to the problem of how the SD2P is preserved by projective tensor products, posed in \cite[Question (b)]{aln}.
In Section~\ref{sectionproje} we study octahedrality of projective tensor
products. In Theorem \ref{teolsumaoctapten} we will prove that
octahedrality is preserved from one of the factors by taking projective
tensor products in presence of the metric approximation property whenever one of the factors in a non-reflexive $L$-embedded Banach space, which provides a partial positive answer to
\cite[Question~4.4]{llr}.

\section{Notation and preliminaries}\label{sectionpreli}
\bigskip

We will only consider real and non-zero Banach spaces and we follow
standard Banach space notation as used in e.g. \cite{alka}. Given a
Banach space $X$ we denote the closed unit ball by $B_X$ and the
closed unit sphere by $S_X$.
The Banach space of bounded linear operators from
$X$ to a Banach space $Y$ is denoted by $L(X,Y)$,
while the subspace of finite rank operators is denoted by $F(X,Y)$.
By $L_1$ we mean the Banach space $L_1[0,1]$.
By $p^*$ we denote the conjugate exponent of $1 \le p \le \infty$
defined by $\frac{1}{p} + \frac{1}{p^*} = 1$.

Let $I$ be the identity operator on a Banach space $X$.
Recall that $X$ has the \emph{Daugavet property} if the equation
\begin{equation*}
  \| I + T \| = 1 + \|T\|
\end{equation*}
holds for every rank one operator $T$ on $X$.
Note that if $X$ has the Daugavet property, then the norms
of both $X$ and $X^*$ are octahedral \cite[Corollary~2.5]{blr0}.

Given two Banach spaces $X$ and $Y$ we will denote by $X \iten Y$ the injective,
and by $X \pten Y$ the projective, tensor product of $X$ and $Y$.
Our main reference for the theory of tensor products
of Banach spaces is \cite{rya}.

A Banach space $X$ has the \emph{diameter two property} (D2P)
if every non-empty relatively weakly open subset of $B_X$
has diameter two. $X$ has the D2P if and only if
the norm of the dual space is \emph{weakly octahedral}.
For the definition of weak octahedrality and
its relation to D2P we refer to \cite{hlp0, lan}.

According to \cite[Definition III.1.1]{hww}, a Banach space $X$ is said to be an \textit{$L$-embedded Banach space} if there exists a subspace $Z\subseteq X^{**}$ such that $X^{**}=X\oplus_1 Z$. Note that from the Principle of Local Reflexivity, non-reflexive $L$-embedded Banach space have an octahedral norm.

In Section~\ref{sectionproje} the theory of almost isometric ideals
will play an important role in our results about octahedrality in
projective tensor products.
Let $Z$ be a subspace of a Banach space $X$.
We say that $Z$ is an \emph{almost isometric ideal} (ai-ideal) in $X$ if
$X$ is locally complemented in $Z$ by almost isometries.
This means that for each $\varepsilon>0$ and for each
finite-dimensional subspace $E\subseteq X$ there exists a linear
operator $T:E\to Z$ satisfying
\begin{enumerate}
\item\label{item:ai-1}
  $T(e)=e$ for each $e\in E\cap Z$, and
\item\label{item:ai-2}
  $(1-\varepsilon) \Vert e \Vert \leq \Vert T(e)\Vert\leq
  (1+\varepsilon) \Vert e \Vert$
  for each $e\in E$,
\end{enumerate}
i.e. $T$ is a $(1+\varepsilon)$ isometry fixing the elements of $E$.
If the $T$'s satisfy only (\ref{item:ai-1}) and the right-hand side of
(\ref{item:ai-2}) we get the well-known
concept of $Z$ being an \emph{ideal} in $X$ \cite{GKS}.

Note that the Principle of Local Reflexivity means that $X$ is an ai-ideal in $X^{**}$
for every Banach space $X$. Moreover, the Daugavet property, octahedrality and
all of the diameter two properties are inherited by ai-ideals
(see \cite{abrahamsen} and \cite{aln2}).

Let $X$ be a Banach space and let $\alpha$ be a tensor norm.
By \cite[Proposition~6.4]{rya}
$X \aten Y$ is a subspace of $X^{**} \aten Y$
for any Banach space $Y$.
A similar argument shows that this
result can be generalised to (ai-)ideals,
i.e. $Z \aten Y$ is a subspace of $X \aten Y$
for any Banach space $Y$ whenever $Z$ is an ideal in $X$.
In Section~\ref{sectionproje} we will need the following
version of this result:

\begin{proposition}\label{subprojeai}
  Let $Z$ be an (ai-)ideal in $X$ and let $Y$ be a Banach space.
  Then $Z \pten Y$ is a subspace of $X \pten Y$.
\end{proposition}

With an extra assumption we even get an ai-ideal.

\begin{proposition}
  Let $X$ and $Y$ be Banach spaces.
  If $L(Y,X^*)$ is norming for $X^{**} \pten Y$,
  then $X \pten Y$ is an ai-ideal in $X^{**} \pten Y$.
\end{proposition}

\begin{proof}
  We have $(X \pten Y)^* = L(X,Y^*)=L(Y,X^*)$
  and $(X^{**} \pten Y)^* = L(X^{**},Y^*)$.

  Define an operator
  $\phi : L(Y,X^*) \to L(X^{**},Y^*)$ by
  $\phi(T) := T^{*}$.
  We have $\langle \phi(T), u \rangle = \langle T,u \rangle$
  for $u \in X \pten Y$ and $\|\phi\| \le 1$,
  thus $\phi$ is a Hahn-Banach extension operator.
  By assumption $\phi(L(Y,X^{*}))$ is norming, so by
  \cite[Proposition~2.1]{aln2} we have that
  $X \pten Y$ is an ai-ideal in $X^{**} \pten Y$.
\end{proof}

Recall that a Banach space $X$ has the \emph{metric approximation property} (MAP)
if there exists a net $(S_\alpha)$ in $F(X,X)$ such that
$S_\alpha x \to x$ for all $x \in X$.
The MAP allows us give examples of spaces where
the above proposition applies.

\begin{proposition}\label{normingaitensor}
  Let $X$ and $Y$ be Banach spaces.
  If either $X^{**}$ or $Y$ has the MAP, then
  $F(Y,X^*) \subset L(Y,X^*)$ is norming for $X^{**} \pten Y$.
  In particular,
  $X \pten Y$ is an ai-ideal in $X^{**} \pten Y$.
\end{proposition}

\begin{proof}
  Let $\varepsilon > 0$.
  Let $u \in X^{**} \pten Y$ and choose
  a representation $u = \sum_{n=1}^\infty x_n^{**} \otimes y_n$
  such that $\sum_{n=1}^\infty \|x_n^{**}\|\|y_n\| > \|u\| - \varepsilon$.
  Choose $N$ such that
  \begin{equation*}
    \sum_{n > N} \|x_n^{**}\|\|y_n\| < \varepsilon.
  \end{equation*}
  Then for all $T \in L(X^{**},Y^*)$ with $\|T\| \le 1$
  \begin{equation*}
    \left|
      \langle T, u \rangle - \sum_{n=1}^{N} Tx_n^{**}(y_n)
    \right|
    \le
    \sum_{n > N} \|Tx_n^{**}\|\|y_n\| < \varepsilon.
  \end{equation*}

  Choose $T \in L(X^{**},Y^*)$ with $\|T\| \le 1$
  such that $\langle T, u \rangle = \|u\|$.

  Assume first that $Y$ has the MAP and assume, with no loss of generality, that $\Vert Tx_n^{**}\Vert\leq 1$ holds for every $n\in\{1,\ldots, N\}$.
  Then there exists a net $(S_\alpha) \subseteq F(Y,Y)$
  such that $\|S_\alpha\| \le 1$ and $S_\alpha y \to y$ for all $y \in Y$.
  Choose $\alpha_0$ large enough so that
  $\|S_{\alpha_0} y_n - y_n\| < \varepsilon/N$ for
  $n \in \{1,\ldots,N\}$.
  Define $T_0 \in F(X^{**},Y^*)$
  by $T_0 := S_{\alpha_0}^* T$.
  Then
  \begin{equation*}
    \left|
      \sum_{n=1}^{N} Tx_n^{**}(y_n) -      
      \sum_{n=1}^{N} T_0 x_n^{**}(y_n)
    \right|
    \le
    \sum_{n=1}^N \|Tx_n^{**}\|\|S_{\alpha_0}y_n - y_n\|
    < \varepsilon.
  \end{equation*}
  
  Similarly, if we assume that $X^{**}$ has the MAP,
  there exists a net $(S_\alpha) \subseteq F(X^{**},X^{**})$
  such that $\|S_\alpha\| \le 1$ and $S_\alpha x^{**} \to x^{**}$ for all $x^{**} \in X^{**}$. Again assume with no loss of generality that $\Vert T^*y_n\Vert\leq 1$ holds for every $n\in\{1,\ldots, N\}$ and choose $\alpha_0$ large enough so that
  $\|S_{\alpha_0} x_n^{**} - x_n^{**}\| < \varepsilon/N$
  for $n \in \{1,\ldots,N\}$.
  Define $T_0 \in F(X^{**},Y^*)$
  by $T_0 :=  T S_{\alpha_0}$.
  Then
  \begin{equation*}
    \left|
      \sum_{n=1}^{N} Tx_n^{**}(y_n) -      
      \sum_{n=1}^{N} T_0 x_n^{**}(y_n)
    \right|
    \le
    \sum_{n=1}^N \|S_{\alpha_0} x_n^{**} - x_n^{**}\|\|T^{*}y_n\|
    < \varepsilon.
  \end{equation*}
  
  So far, in both cases we have found $T_0 \in F(X^{**},Y^*)$
  such that
  \begin{equation*}
    |\langle T, u\rangle - \langle T_0, u \rangle|
    < 3\varepsilon.
  \end{equation*}
  Next we use \cite[Theorem~2.5]{OP3} to find
  $T_1 \in F(Y,X^*) = F(X,Y^*)$ such that
  $\|T_1\| \le 1+\varepsilon$ and $T_1^* x_n^{**} = T_0 x_n^{**}$
  for $n \in \{1,\ldots,N\}$.
  This implies that
  \begin{align*}
    |\langle T, u\rangle - \langle T_1, u \rangle|
    &\le
    |\langle T, u\rangle - \langle T_0, u \rangle|
    +
    |\langle T_0, u\rangle - \langle T_1, u \rangle| \\
    &< 3\varepsilon + 2\varepsilon = 5\varepsilon.
  \end{align*}
  Hence we have $\langle T_1, u \rangle > \|u\| - 5\varepsilon$.
  Since $\varepsilon > 0$ was arbitrary we get that
  $F(X,Y^*)$ is norming for $X^{**} \pten Y$.
\end{proof}

Related to almost isometric ideals is the notion
of finite representability.
In Section~\ref{sectioninje} we shall need a characterisation of when
a separable Banach space is finitely representable in $\ell_1$.
The following lemmata are probably well-known,
but we include their proofs for easy reference.

\begin{lemma}\label{lem:sepfinreprLp-isometric}
  Let $\nu$ be a $\sigma$-additive measure.
  If a separable Banach space $X$ is finitely representable
  in $L_p(\nu)$, $1 \le p < \infty$,
  then it is isometric to a subspace of $L_p[0,1]$.
\end{lemma}

\begin{proof}
  By \cite[Proposition~11.1.12]{alka} $X$
  is isometric to a subspace of an ultrapower
  $Y = (L_p(\nu))_{\mathcal{U}}$ of $L_p(\nu)$
  for a nonprincipal ultrafilter $\mathcal{U}$.
  But $Y$ is isometric to $L_p(\mu)$ for some measure
  $\mu$ \cite[Theorem~3.3]{Hei1}.
  Since any separable subspace of $L_p(\mu)$ is
  isometric to a subspace of some separable $L_p(\mu_1)$
  \cite[Proposition~III.A.2]{Woj91},
  which is isometric to a subspace
  of $L_p[0,1]$ \cite[pp.~14--15]{JohLin},
  the lemma follows.
\end{proof}

\begin{lemma}\label{lemma:char_fin_repr_L1}
  Let $X$ be a separable Banach space.
  The following are equivalent:
  \begin{enumerate}
  \item\label{item:L1-1}
    $X$ is finitely representable in $L_1$.
  \item\label{item:L1-0}
    $X$ is finitely representable in $\ell_1$.
  \item\label{item:L1-2}
    $X$ is isometric to a subspace of $L_1$.
  \item\label{item:L1-3}
    For all $\varepsilon > 0$ there exists
    a $(1+\varepsilon)$ isometry from $X$
    into $L_1$.
  \end{enumerate}
  These statements are implied by
  \begin{enumerate}[resume]
  \item\label{item:L1-4}
    For all $\varepsilon > 0$ there exists
    a $(1+\varepsilon)$ isometry from $X$
    into $\ell_1$.
  \end{enumerate}
  If $X$ is finite-dimensional, then all the statements
  are equivalent.
\end{lemma}

\begin{proof}
  (\ref{item:L1-1}) $\Rightarrow$ (\ref{item:L1-0})
  since $L_1$ is finitely representable in $\ell_1$
  \cite[Proposition~11.1.7]{alka}.
  (\ref{item:L1-0}) $\Rightarrow$ (\ref{item:L1-2}) by
  Lemma~\ref{lem:sepfinreprLp-isometric}.
  (\ref{item:L1-2}) $\Rightarrow$ (\ref{item:L1-3}),
  (\ref{item:L1-3}) $\Rightarrow$ (\ref{item:L1-1})
  and
  (\ref{item:L1-4}) $\Rightarrow$ (\ref{item:L1-0})
  are all trivial.

  If $X$ is finite-dimensional, then
  (\ref{item:L1-2}) $\Rightarrow$ (\ref{item:L1-4})
  by finite representability of $L_1$ in $\ell_1$.
\end{proof}

\section{Octahedrality in injective tensor products}\label{sectioninje}
\bigskip

The authors of \cite{hlp} introduced a new notion of octahedrality.
The norm of a Banach space $X$ is \emph{alternatively octahedral} if, for every
$x_1,\ldots,x_n \in S_X$ and $\varepsilon > 0$, there is
a $y \in S_X$ such that
\begin{equation*}
  \max\{ \|x_i+y\|, \|x_i - y\| \} > 2 - \varepsilon
  \qquad \mbox{for all} \; i \in \{1,\ldots,n\}.
\end{equation*}
This norm condition
implies that there exist $x_1^*,\ldots,x_n^* \in S_{X^*}$
such that
\begin{equation*}
  |x_i^*(x_i)| > 1-\varepsilon
  \quad \mbox{and} \quad
  |x_i^*(y)| > 1-\varepsilon\quad \mbox{for every } i \in \{1,\ldots,n \}.
\end{equation*}

It is known that the norm of $X$ is octahedral if, and only
if, for every $x_1,\ldots, x_n\in S_X$ and
$\varepsilon>0$ there exists $y\in S_X$ such that
$\Vert x_i+y\Vert>2-\varepsilon$ for all $i\in\{1,\ldots, n\}$
(see \cite[Proposition~2.1]{hlp0}).
Consequently, octahedrality implies alternative octahedrality.
However, the converse does not hold.

\begin{example}
  It is not difficult to see that $c_0$ and $\ell_\infty$ do not have
  an octahedral norm. However, the norms of these spaces are
  alternatively octahedral.
  To see this consider elements $x_1,\ldots, x_n$ of norm one
  and let $i_1,\ldots,i_m$ be distinct indices where these elements
  (almost) attain their norm. The norm one element
  $y = e_{i_1} + e_{i_2} + \cdots + e_{i_m}$ does the job.
\end{example}

In \cite{llr} it is shown that if $X$ and $Y$ are Banach
spaces whose norms are octahedral then the norm of $X \iten Y$
is also octahedral. The following proposition is similar to
\cite[Theorem~2.1]{hlp} and improves \cite[Theorem~2.2]{llr}.

\begin{proposition}\label{prop:operator-OHandAltOH}
  Let $X$ and $Y$ be Banach spaces and
  $H$ a subspace of $L(X^*,Y)$ containing $X \otimes Y$
  such that every $T \in H$ is weak$^*$-weakly continuous.
  If the norm of $X$ is alternatively octahedral and the norm of $Y$
  is octahedral, then the norm of $H$ is octahedral.
\end{proposition}

\begin{proof}
  Let $T_1,\ldots,T_n \in S_H$ and $\varepsilon > 0$.
  For each $i \in \{1,\ldots,n\}$ find $y_i^* \in S_{Y^*}$ such that
  $\|T_i^* y_i^*\| > 1 - \varepsilon$. Note that $T_i^*y_i^* \in X$
  for all $i\in\{1,\ldots, n\}$ since $H$ consists of
  weak$^*$-weakly continuous operators.
  Since the norm of $X$ is alternatively octahedral there exist
  $x_1^*,\ldots,x_n^* \in S_{X^*}$
  and $w \in S_X$ such that $|x_i^*(w)| > 1 - \varepsilon$ and
  \begin{equation*}
    |x_i^*(T_i^* y_i^*)| > \|T_i^* y_i^*\| (1-\varepsilon)
    > (1-\varepsilon)^2
  \end{equation*}
  holds for every $i\in\{1,\ldots, n\}$. We may assume that
  $x_i^*(T_i^* y_i^*) > 0$ for all $i \in \{1,\ldots,n\}$.
  Define $\gamma_i := \sgn x_i^*(w)$.

  Let $F = \linspan\{T_i x_i^* : i \in \{1,\ldots,n\} \} \subset Y$.
  Use octahedrality and \cite[Theorem~3.21]{lan} to find
  $z \in S_Y$ and $z_i^* \in Y^*$, $i \in \{1,\ldots, n\}$,
  such that $z_i^*(T_i x_i^*) = y_i^*(T_i x_i^*)$,
  $z_i^*(z) = \gamma_i$ and $\|z_i^*\| \le 1 + \varepsilon$ holds for every $i$.

  Define $S := w \otimes z \in X \otimes Y$. We have $S \in S_H$
  and, for each $i\in\{1,\ldots, n\}$, it follows that
  \begin{align*}
    \|T_i + S\|
    &\ge \frac{1}{1+\varepsilon}
      z_i^*(T_i x_i^* + Sx_i^*)
    =  \frac{1}{1+\varepsilon}
      (y_i^*(T_i x_i^*) + x_i^*(w)z_i^*(z)) \\
    &=  \frac{1}{1+\varepsilon}
      (y_i^*(T_i x_i^*) + |x_i^*(w)|)
      > \frac{2-3\varepsilon+\varepsilon^2}{1+\varepsilon}
      > 2 - 5 \varepsilon.
  \end{align*}
  Hence we conclude that the norm of $H$ is octahedral.
\end{proof}

Throughout the rest of this section we study whether the norm of
$X\iten Y$ is octahedral when we assume that the norm of only one of
the factors is octahedral. For this, we shall begin by giving some
positive results for the Banach spaces $\ell_1$ and $L_1$, which have
an octahedral norm.

\begin{theorem}\label{lema2dim}
  Let $X$ be a Banach space. Then:
  \begin{enumerate}
  \item \label{lema2dim1} If $X$ is $(1+\varepsilon)$ isometric to a
    subspace of $\ell_1$, then the norm of $L(X,\ell_1)$ is octahedral.
  \item \label{lema2dim2}  If $X$ is $(1+\varepsilon)$ isometric to a
    subspace of $L_1$, then the norm of $L(X,L_1)$ is octahedral.
  \end{enumerate}
\end{theorem}

\begin{proof}
  (\ref{lema2dim1}).
  Let $\varepsilon>0$ and
  $\psi: X \to \ell_1$ be a $(1+\varepsilon)$ isometry.
  Let $T_1,\ldots, T_n\in S_{L(X,\ell_1)}$
  and, for every $i\in\{1,\ldots, n\}$, pick $x_i\in S_{X}$
  such that $\Vert T_i(x_i)\Vert>1-\varepsilon$.

  Let $P_k$ be the projection on $\ell_1$ onto the first
  $k$ coordinates. Choose $k \in \mathbb{N}$
  so that $\|P_k(T_i(x_i)) - T_i(x_i)\| < \varepsilon$
  and $\|P_k(\psi(x_i)) - \psi(x_i)\| < \varepsilon$
  for every $i \in \{1,\ldots, n\}$.

  Let $\varphi_k: \ell_1 \to \ell_1$ be
  the shift operator defined by
  \begin{equation*}
    \varphi_k(x)(n):=
    \begin{cases}
      0 & \mbox{if }n\leq k,\\
      x(n-k) & \mbox{if }n>k.
    \end{cases}
  \end{equation*}
  Define $S:=\varphi_k \circ P_k \circ \psi$.
  Now, as $P_k(T_i(x_i))$ and $S(x_i)$ have disjoint support, we have that
  \begin{align*}
    \Vert T_i + S\Vert
    &\ge
      \Vert P_k T_i(x_i)\Vert - \varepsilon + \Vert P_k(\psi((x_i))) \Vert \\
    &\ge \Vert T_i(x_i)\Vert + \|\psi(x_i)\| - 3\varepsilon
      >
      2-5\varepsilon,
  \end{align*}
  so we are done.

  (\ref{lema2dim2}).
  Define $A:=[0,1]$.
  Let $T_1,\ldots, T_n\in S_{L(X, L_1)}$ and $\varepsilon>0$.
  By assumption there exists $x_i\in S_{X}$ such that
  $\Vert T_i(x_i) \Vert = \int_A \vert T_i(x_i) \vert
  > 1 - \frac{\varepsilon}{2}$ for all $i\in\{1,\ldots,n\}$.
  Pick  a closed interval $I\subseteq A$ such that
  $\int_I \vert T_i(x_i)\vert<\frac{\varepsilon}{2}$ holds for each
  $i\in\{1,\ldots, n\}$.

  By assumption and Lemma~\ref{lemma:char_fin_repr_L1} there exists
  a linear isometry $T:X \to L_1$.
  Let $\phi: I\to A$ be an increasing and affine bijection.
  Define $S_I:L_1\to L_1$ by the equation
  \begin{equation*}
    S_I(f)=(f\circ \phi)\phi' \chi_I \quad \mbox{ for all }f\in L_1,
  \end{equation*}
  where $\chi_I$ denotes the characteristic function on the interval $I$.
  Note that $S_I$ is a linear isometry because of the change
  of variable theorem.
  Indeed
  \begin{equation*}
    \Vert S_I(f)\Vert = \int_I \vert (f\circ \phi)\phi'\vert
      = \int_{\phi(I)}\vert f\vert
      = \int_A\vert f\vert
      = \Vert f\Vert\ \mbox{for all } f\in L_1.
  \end{equation*}

  Define $G:=S_I\circ T$, which is a linear isometry such that
  $\supp(G(f))\subseteq I$ for all $f\in L_1$. Given $i\in\{1,\ldots, n\}$, we have
  \begin{equation*}
    \Vert T_i+G\Vert
      \geq \Vert T_i(x_i)+G(x_i)\Vert
      = \int_{A\setminus I} \vert T_i(x_i)\vert + \int_I \vert T_i(x_i) + G(x_i)\vert.
  \end{equation*}
  Now
  \begin{equation*}
    \int_{A\setminus I} \vert T_i(x_i)\vert
      = \Vert T_i(x_i)\Vert - \int_I \vert T_i(x_i)\vert
      > 1-\varepsilon.
  \end{equation*}
  Moreover
  \begin{equation*}
    \int_I \vert T_i(x_i) + G(x_i)\vert
      \geq \int_I \vert G(x_i)\vert - \vert T_i(x_i)\vert
      > \int_I \vert G(x_i)\vert - \frac{\varepsilon}{2}.
  \end{equation*}
  Finally note that, as $\supp(G(x_i))\subseteq I$, we have
  $\int_I \vert G(x_i)\vert = \Vert G(x_i)\Vert =\Vert x_i\Vert = 1$.
  Consequently
  \begin{equation*}
    \Vert T_i+G\Vert>2-2\varepsilon.
  \end{equation*}
  As $\varepsilon$ was arbitrary we conclude that the norm of
  $L(X,L_1)$ is octahedral, as desired.
\end{proof}

From here we can conclude the following result.

\begin{corollary}\label{2dimell1OH}
  If $X$ is a 2-dimensional Banach space,
  then the norms of both $\ell_1\iten X=L(c_0,X)$ and $L_1 \iten X$ are octahedral.
\end{corollary}

\begin{proof}
  We have that $X^*$ is isometric to a subspace of $L_1$ \cite[Corollary~1.4]{dor}.
  From Lemma~\ref{lemma:char_fin_repr_L1} we see that Theorem~\ref{lema2dim} applies.
\end{proof}

Note that the above corollary improves \cite[Proposition~2.3]{hlp}, where
the authors show that the norm of $L(c_0,\ell_p^2)$ is octahedral for
every $1\leq p\leq \infty$. Dualising we get the following result,
which improves \cite[Proposition~2.10]{llr} for two
dimensional Banach spaces.

\begin{corollary}
  If $X$ is a 2-dimensional Banach space,
  then $c_0\pten X$ has the SD2P.
\end{corollary}

Next we give more examples of finite-dimensional Banach spaces for
which the norm of its injective tensor product with $\ell_1$ and $L_1$
are octahedral.

\begin{proposition}\label{posingeq3}
  Let $n\geq 3$ be a natural number and $2\leq p\leq \infty$.
  Then the norms of both $\ell_1\iten \ell_p^n$ and $L_1\iten \ell_p^n$ are octahedral.
\end{proposition}

\begin{proof}
  We know that $\ell_{p^*}$ is isometric to a subspace of $L_1$
  \cite[Theorem~6.4.19]{alka} which in turn contains $\ell_{p^*}^n$
  isometrically.
  From Lemma~\ref{lemma:char_fin_repr_L1} we see that Theorem~\ref{lema2dim}
  applies and shows that the norm of $Y \iten \ell_p^n =
  L(\ell_{p^*}^n,Y)$ is octahedral for $Y = \ell_1$ and $Y = L_1$.
\end{proof}

In fact, an infinite-dimensional version of the previous result also works.

\begin{proposition}\label{lpohinfidime}
  Let $2\leq p< \infty$. Then:
  \begin{enumerate}
  \item \label{lpohinfidime1} Given a closed subspace $H$ of
    $L(\ell_{p^*},\ell_1)$ containing $\ell_p\otimes \ell_1$, then the
    norm of $H$ is octahedral.
  \item \label{lpohinfidim2} Given a closed subspace $H$ of
    $L(\ell_{p^*},L_1)$ containing $\ell_p\otimes L_1$, then the norm of
    $H$ is octahedral.
  \end{enumerate}
\end{proposition}

\begin{proof}
  (\ref{lpohinfidime1}).
  We proceed as in Theorem~\ref{lema2dim}.
  Given $T_1,\ldots, T_n\in S_H$ and $\varepsilon>0$
  we start by choosing, for every
  $i\in\{1,\ldots, n\}$, an element $x_i\in S_{\ell_{p^*}}$ such that
  $\| T_i(x_i) \| > 1 - \varepsilon$.
  Find $m \in \mathbb{N}$ such that $\|P_m(x_i) - x_i\| < \varepsilon$,
  where $P_m$ is the projection onto the first $m$ coordinates.
  Since $\ell_{p^*}$ is finitely representable in $\ell_1$
  there exists a $(1+\varepsilon)$ isometry $T: P_m(\ell_{p^*}) \to \ell_1$.
  The operator $\psi := T \circ P_m$ is then well-defined
  and using this $\psi$ we define $S := \varphi_k \circ P_k \circ \psi$
  as in the proof of Theorem~\ref{lema2dim}.
  Note that $S \in \ell_p \otimes \ell_1 \subseteq H$
  since $P_m$ has finite rank.
  Similar calculations to the ones in
  Theorem~\ref{lema2dim} conclude the proof.

  The proof of (\ref{lpohinfidim2}) is similar, but in this
  case we can use an isometry $T: \ell_{p^*} \to L_1$.
\end{proof}

The above results can be seen as sufficient conditions to get
octahedrality in injective tensor products spaces. Now we turn to
analyse some necessary conditions.

\begin{lemma}\label{lemacontraopera}
  Let $X$ and $Y$ be Banach spaces and assume that $Y^*$ is uniformly convex.
  Assume also that there exists a closed subspace $H$ of $L(Y^*,X)$
  such that $X \otimes Y \subseteq H$
  and that the norm of $H$ is octahedral.
  Then $Y^*$ is finitely representable in $X$.
\end{lemma}

\begin{proof}
  Recall that the modulus of uniform convexity of $Y^*$ is defined by
  \begin{equation*}
    \delta(\varepsilon)
    = \inf
    \left\{
      1 - \left\| \frac{f+g}{2} \right\|
      : f,g \in B_{Y^*}, \|f-g\| \ge \varepsilon
    \right\}.
  \end{equation*}
  Note that if $f,g \in B_{Y^*}$ satisfy $f(y) > 1-\delta(\varepsilon)$
  and $g(y) > 1-\delta(\varepsilon)$, for some $y \in S_Y$,
  then $\|f-g\| < \varepsilon$.

  Let $\varepsilon > 0$ and choose $\nu > 0$ so small
  that $(1+\nu)(1-3\nu)^{-1} < 1+\varepsilon$.
  Pick $0 < \eta < \nu/2$ such that
  $\delta(\eta) < \nu/2$.

  Let $F \subseteq Y^*$ be a finite-dimensional subspace.
  Pick a $\nu$-net $(f_i)_{i=1}^n$ for $S_F$.
  Choose $y_i \in S_Y$ such that $f_i(y_i) = 1$.

  Let $x \in S_X$. By assumption the norm of H is octahedral, so there exists
  a $T \in S_H$ such that
  \begin{equation*}
    \|y_i \otimes x + T\| > 2 - \delta(\eta)
  \end{equation*}
  holds for every $i \in \{1,...,n\}$.

  We want to show that $F$ is $(1+\varepsilon)$ isometric
  to a subspace of $X$.
  We have $\|T(f)\| \le \|f\|$ since $T$ has norm one.
  For $y_i$ we choose $\varphi_i \in S_{Y^*}$ such that
  \begin{equation*}
    \|\varphi_i(y_i)x + T(\varphi_i)\| > 2 - \delta(\eta).
  \end{equation*}
  By the triangle inequality $|\varphi_i(y_i)| > 1 - \delta(\eta)$
  and $\|T(\varphi_i)\| > 1 - \delta(\eta)$.
  We may assume that $\varphi_i(y_i) > 1 - \delta(\eta)$.
  Since $f_i(y_i) = 1$ we get $\|f_i - \varphi_i\| < \eta <
  \nu/2$.
  We also get
  \begin{equation*}
    \|T(f_i)\| \ge \|T(\varphi_i)\| - \|T\|\|f_i - \varphi_i\|
    > 1 - \delta(\eta) - \frac{\nu}{2} > 1 - \nu.
  \end{equation*}
  From \cite[Lemma~11.1.11]{alka}
  we see that $T$ restricted to $F$ is
  a $(1+\varepsilon)$ isometry.
\end{proof}

Using the above lemma we get the following result.

\begin{theorem}\label{teocentracontraeje-v2}
  For every $1 < p < 2$ and every natural number $n \ge 3$ the norms
  of both $\ell_1 \iten \ell_p^n$ and
  $L_1 \iten \ell_p^n$ fail to be octahedral.
\end{theorem}

\begin{proof}
  From \cite[Theorem~1.5]{dor} we see that
  $\ell_{p^*}^n$ is not isometric to a subspace of $L_1[0,1]$.
  Combining Lemma~\ref{lemma:char_fin_repr_L1} and
  Lemma~\ref{lemacontraopera} we get the desired conclusion.
\end{proof}

Now, if we dualise Theorem~\ref{teocentracontraeje-v2} and use
\cite[Proposition~5.33]{rya}, we get the following corollary,
which gives a negative answer to an open
problem posed in \cite[p.~177]{blr} as well as in \cite{hlp}.

\begin{corollary}\label{contraproyec}
  Let $2<p<\infty$ and $n\geq 3$. Then neither $\ell_\infty\pten
  \ell_p^n$ nor $L_\infty\pten \ell_p^n$ enjoy the SD2P.
\end{corollary}

The above theorem together with
Proposition~\ref{prop:operator-OHandAltOH} and
Proposition~\ref{lpohinfidime} allow us
to give the following characterisation of octahedrality when dealing
with classic sequence spaces.

\begin{theorem}\label{summary}
  Let $1\leq p\leq \infty$ and let $X$ be either $L_1$ or $\ell_1$. Then:
  \begin{enumerate}
  \item If $H$ is a closed subspace of $L(\ell_{p^*},X)$ which contains
    $\ell_p\otimes X$, then the norm of $H$ is octahedral if, and only
    if, $2\leq p$ or $p=1$.
  \item If $H$ is a closed subspace of $L(\ell_1,X)$ which contains
    $c_0\otimes X$, then the norm of $H$ is octahedral.
  \item If $n$ is a natural number and $H$ is a closed subspace of
    $L(\ell_{p^*}^n,X)$ which contains $\ell_p^n \otimes X$, then the
    norm of $H$ is octahedral if, and only if,
    either $n\leq 2$ or if $n>2$ and $2\leq p$ or $p=1$.
  \end{enumerate}
\end{theorem}

Theorem~\ref{teocentracontraeje-v2} and Corollary~\ref{contraproyec}
also allow us to shed light on a number of results and questions
from the literature.

\begin{rem}
  In \cite[Question~(b)]{aln} it is asked how diameter two
  properties are preserved by tensor products.
  We can now provide a complete answer to this question for the SD2P
  in the projective case.
  The SD2P is preserved from both factors,
  by \cite[Corollary~3.6]{blr}, but not in general from one of them,
  by Corollary~\ref{contraproyec}.

  In \cite[Question~4.1]{llr} it is asked whether octahedrality is preserved
  by injective tensor products just from one of the factors.
  Theorem~\ref{teocentracontraeje-v2} gives a negative answer
  to this question.
\end{rem}

\begin{rem}
  Note that $L_\infty$ as well as $\ell_\infty$ have an
  infinite-dimensional centralizer \cite[Example~I.3.4.(h)]{hww}.
  From \cite[Corollary~3.8]{blr} and Corollary~\ref{contraproyec}
  we see that, given two Banach spaces $X$ and $Y$,
  it is not enough to assume that $X$ has an infinite-dimensional
  centralizer to ensure that $X\pten Y$ has the SD2P.
  But both $L_\infty$ and $\ell_\infty$ are isometric
  to $C(K)$ spaces so $L_\infty \pten Y$ and $\ell_\infty \pten Y$
  do have the D2P for any $Y$ by
  \cite[Proposition~4.1]{abr}.

  For some spaces we can say even more.
  Let $Y = \ell_p$ or $Y = \ell_p^n$ with $n$ a natural number
  and $1 < p < \infty$. By \cite[Proposition~5.33]{rya}
  we have $(\ell_1 \iten Y)^* = \ell_\infty \pten Y^*$.
  By \cite[Theorem~2.7]{hlp0} we get that
  the bidual $(\ell_1 \iten Y)^{**}$ is weakly octahedral.
  But, for $1 < p < 2$ and $n \ge 3$, $\ell_1 \iten Y$
  is not octahedral, by Theorem~\ref{teocentracontraeje-v2},
  hence $\ell_\infty \pten Y^*$ does not even have
  the $w^*$-SD2P (see e.g. \cite[Theorem~2.1]{blr0} or \cite[Theorem~2.2]{hlp0}).
\end{rem}

\begin{rem}\label{remdaugaytensor}
  Our results also give natural examples of
  tensor products failing the Daugavet property.

  By Theorem~4.2 and Corollary~4.3 in \cite{kkw} there exists
  a two dimensional complex Banach space $E$ such that
  both $L_1^{\mathbb C} \iten E$ and
  $L_\infty^{\mathbb C} \pten E^*$ fail the Daugavet property.
  Note that both real and complex $L_1$ and $L_\infty$
  have the Daugavet property.

  A Banach space with the Daugavet property has the SD2P
  and an octahedral norm
  (\cite[Corollary~2.5]{blr0} and \cite[Theorem~4.4]{aln}).
  We can thus improve the above mentioned results of \cite{kkw}
  by giving examples of (real) Daugavet spaces such that their
  tensor product fail to be octahedral
  or fail to have the SD2P.
  By Theorem~\ref{summary},
  $L_1\iten \ell_p^n$ does not have an octahedral norm for $1<p<2$ and $n\geq 3$, and
  by Corollary~\ref{contraproyec},
  $L_\infty \pten \ell_p^n$ does not have the SD2P for $2 < p < \infty$ and $n \ge 3$.
\end{rem}

\section{Octahedrality in projective tensor products}\label{sectionproje}

\bigskip

Given two Banach spaces $X$ and $Y$, no octahedrality
assumption is needed on $X$ or $Y$ in order for $X\pten Y$
to have an octahedral norm. Indeed, it follows from
\cite[Corollary~III.1.3]{hww} and the Principle of Local Reflexivity that
the norm of $\ell_2 \pten \ell_2$ is octahedral in spite of the fact that
$\ell_2$ is a Hilbert space. On the other hand, if we assume that one of the
factors is finite-dimensional, then the octahedrality of $X\pten Y$ forces
the other factor to have an octahedral norm.

\begin{proposition}\label{necefinidimen}
Let $X$ and $Y$ be Banach spaces. Assume that $Y$ is finite-dimensional and that $X\pten Y$ has an octahedral norm. Then $X$ has an octahedral norm.
\end{proposition}

\begin{proof}
  Pick $x_1,\ldots, x_n\in S_X$ and $\varepsilon>0$.
  Since $X\pten Y$ has an octahedral norm and $Y$ is finite-dimensional
  we can find by \cite[Proposition 3.2]{pr} $u\otimes v\in S_X\otimes S_Y$ such that
  \begin{equation*}
    \Vert x_i\otimes y+u\otimes v\Vert\geq 2-\varepsilon
  \end{equation*}
  holds for all $y\in Y$ and $i\in\{1,\ldots, n\}$.
  Now, given $i\in\{1,\ldots, n\}$, we have
  \begin{equation*}
    2-\varepsilon \leq
    \Vert x_i\otimes v + u\otimes v\Vert
    =
    \Vert x_i+u\Vert\Vert v\Vert
    =
    \Vert x_i+u\Vert.
  \end{equation*}
  Hence, $X$ has an octahedral norm, as desired.
\end{proof}

\begin{lemma}\label{lemma:alapropPaiideal}
  Let $X$ and $Z$ be Banach spaces.
  If $Z$ is an ai-ideal in $X$ and, for every $z_1,\ldots, z_n\in S_Z$ there exists
   $v \in S_X$ such that
  \begin{equation*}
    \|z_i + v\| = \|z_i\| + \|v\| \quad \mbox{for all}\ i\in\{1,\ldots, n\},
  \end{equation*}
  then the norm of $Z$ is octahedral.
\end{lemma}

\begin{proof}
  Let $z_1,\ldots,z_n \in S_Z$, $\varepsilon > 0$ and $v$ as in the hypothesis of the lemma.
  Define
  $E := \linspan\{v, z_1,\ldots,z_n\}$.
  Find $T : E \to Z$ such that
  $T(e) = e$ for all $e \in E \cap Z$
  and
  \begin{equation*}
    (1-\varepsilon)\|e\| \le \|T(e)\| \le (1+\varepsilon)\|e\|
    \quad \mbox{for all}\ e \in E.
  \end{equation*}
  Let $q = \frac{T(v)}{\|T(v)\|} \in S_Z$.
  We have
  \begin{align*}
    \|z_i + q\| &\ge \|z_i + T(v)\| - \varepsilon
    = \|T(z_i + v)\| - \varepsilon \\
    &\ge (1-\varepsilon)\|z_i + v\| - \varepsilon
    = 2 - 3\varepsilon,
  \end{align*}
  which means that the norm of $Z$ is octahedral.
\end{proof}

The following theorem provides a partial positive answer to
\cite[Question~4.4]{llr}, where it is asked whether octahedrality
is preserved by taking projective tensor products from one of the factors.

\begin{theorem}\label{teolsumaoctapten}
  Let $X$ be a non-reflexive $L$-embedded space and $Y$ be a Banach
  space. If either $X^{**}$ or $Y$ has the MAP then $X\pten Y$ has an
  octahedral norm.
\end{theorem}

\begin{proof}
  Since $X$ is a non-reflexive $L$-embedded Banach space then
  $X^{**}=X\oplus_1 Z$ for some non-zero subspace $Z$ of $X^{**}$,
  hence $X^{***}=X^*\oplus_\infty Z^*$.

  Let $u \in S_Z$, $y \in S_X$, and $y^* \in S_{Y^*}$
  such that $y^*(y) = 1$ and define $v = u \otimes y$.
  Denote by $X_u = \linspan\{X,u\} = X \oplus_1 \mathbb{R}$.
  By the triangle inequality
  $\|z + v\| \le \|z\| + \|v\|$ in $X_u \pten Y$
  for all $z \in X \pten Y$.
  First we will show that we in fact have equality here.

  To this aim let $z\in X\pten Y$ and pick $T\in S_{L(X,Y^*)}$
  such that $\langle T, z \rangle=\|z\|$.
  Define $\bar{T} : X_u \to Y^*$ by the equation
  \begin{equation*}
    \bar{T}(x+\lambda u)=T(x)+\lambda y^*.
  \end{equation*}
  We claim that $\|\bar{T}\|\leq 1$. Indeed, given an arbitrary
  $x + \lambda u \in X_u$, one has
  \begin{equation*}
    \| \bar{T}(x+\lambda u) \|
    = \| T(x) + \lambda y^* \|
    \leq \|T(x)\| + |\lambda|
    \leq \|x\| + |\lambda|
    = \| x + \lambda u \|.
  \end{equation*}
  Consequently, it follows that
  \begin{equation*}
    \Vert z + v \Vert
    \geq \langle \bar{T}, z + v \rangle
    = \langle T, z \rangle + \langle \bar{T}, v\rangle
    = \|z\| + y^*(y)
    = \|z\| + 1
    = \|z\| + \|v\|.
  \end{equation*}

  We have that $X_u^*$ is isometric to $X^*\oplus_\infty\mathbb R$,
  which is an isometric subspace of $X^*\oplus_\infty Z^*=X^{***}$.
  This implies the existence of a Hahn-Banach operator
  $\varphi:X_u^* \to X^{***}$
  hence $X_u$ is an ideal in $X^{**}$
  \cite[Th{\'e}or{\`e}me~2.14]{Fak}.
  By Proposition~\ref{subprojeai} we
  conclude that $X_u\pten Y$ is an isometric subspace of
  $X^{**}\pten Y$, so
  \begin{equation*}
    \|z + v\|_{X^{**}\pten Y}
    = 1 + \| z\|_{X^{**}\pten Y}
  \end{equation*}
  holds for every $z \in X \pten Y$.
  By Proposition~\ref{normingaitensor}, $X\pten Y$ is an ai-ideal in
  $X^{**}\pten Y$, so Lemma~\ref{lemma:alapropPaiideal} finishes the
  proof.
\end{proof}

For general Banach spaces $X$ and $Y$
the question of whether $X \pten Y$ has an octahedral
norm whenever $X$ and/or $Y$ do remains open. We note
that it is enough to consider separable Banach spaces
to answer this question. This follows by using
\cite[Theorem~1.5]{abrahamsen} and Proposition~\ref{subprojeai}.

\end{document}